\documentclass[12pt,reqno]{amsart}
\usepackage[all,2cell]{xy}

\usepackage{amssymb}

\usepackage[all]{xy}

\newtheorem{theorem}{Theorem}[section]
\newtheorem{proposition}[theorem]{Proposition}
\newtheorem{lemma}[theorem]{Lemma}
\newtheorem{corollary}[theorem]{Corollary}

\theoremstyle{definition}
\newtheorem{definition}[theorem]{Definition}
\newtheorem{remark}[theorem]{Remark}
\newtheorem{example}[theorem]{Example}

\numberwithin{equation}{section}

\DeclareMathOperator*{\image}{Im}

\newcommand{\ZZ}{\mathbb{Z}}

\newcommand{\FF}{\mathbb{F}}
\newcommand{\GG}{\mathbb{G}}

\newcommand{\scrC}{\mathcal{C}}

\newcommand{\scrf}{\mathfrak{f}}
\newcommand{\frakF}{\mathfrak{F}}
\newcommand{\Vect}{\rm Vect}

\begin{document}

\title{Monodromy of stratified vector bundles}

\author[I. Biswas]{Indranil Biswas}

\address{Department of Mathematics, Shiv Nadar University, NH91, Tehsil
Dadri, Greater Noida, Uttar Pradesh 201314, India}

\email{indranil.biswas@snu.edu.in, indranil29@gmail.com}

\author[M. Kumar]{Manish Kumar}

\address{Statistics and Mathematics Unit, Indian Statistical Institute,
Bangalore 560059, India}

\email{manish@isibang.ac.in}

\author[A.J. Parameswaran]{A. J. Parameswaran}

\address{Kerala School of Mathematics, Kunnamangalam PO, Kozhikode, Kerala, 673571, India}

\email{param@ksom.res.in}



\begin{abstract}
We explore the interconnections between the monodromy group of stratified bundles on a smooth
projective variety $X$ and the monodromy of the strongly semistable vector bundles $V$ on $X$
such that $c_1(V)$ and $c_2(V)$ are numerically trivial.
\end{abstract}

\maketitle

\section{Introduction}

Let $k$ be an algebraically closed field of characteristic $p$, with $p\,>\, 0$. Take an irreducible
smooth projective variety $X$ defined over $k$. A stratified vector bundle on $X$ is an
${\mathcal O}_X$--coherent ${\mathcal D}_X$--module (\cite{Gi}, \cite{Sa}) where  ${\mathcal D}_X$ denote the sheaf of differential operators,
in the sense of Grothendieck, on $X$ (\cite{Gr}, \cite{BO}).

Let $F_X\, :\, X\, \longrightarrow\, X$ be the absolute Frobenius morphism for $X$. So for any vector bundle
$V$ on $X$, the pullback $F^*_X V\, \longrightarrow\, X$ is the subbundle of
$V^{\otimes p}$ defined by $\{v^{\otimes p}\, \in\, V^{\otimes p}\,\,\big\vert\,\, v\, \in\, 
V\}$. A $F$--divisible vector bundle on $X$ is a sequence of vector bundles $\{E_i\}_{i\geq 
0}$ on $X$ indexed by the nonnegative integers together with an isomorphism $E_i\, 
\longrightarrow\, F^*_X E_{i+1}$ for every $i\, \geq\, 0$ (\cite{Gi}, \cite{Sa}).

By a result of Katz, there is a natural equivalence of categories between the stratified vector bundles and the
$F$--divisible vector bundles. The underlying vector bundle for the stratified vector bundle 
corresponding to a $F$--divisible vector bundle $\{E_i\}_{i\geq 0}$ is $E_0$ (see \cite{Gi},
\cite{Sa}). Similarly, the rank of an $F$--divisible vector bundle
$\{E_i\}_{i\geq 0}$ is the common rank of all $E_i$.

Let $\Vect^{str}(X)$ be the category of $F$--divisible vector bundles on $X$, which, as 
mentioned above, is equivalent to the category of stratified vector bundles on $X$. 
Henceforth, by a stratified vector bundle we will mean a $F$--divisible vector bundle.

The category of finite dimensional $k$--vector spaces will be denoted by $\Vect(k)$.
Fix a closed point $x$ of $X$. We have the fiber functor
\begin{equation}\label{e1}
\omega_x\ :\ {\rm Vect}^{str}(X)\ \longrightarrow\ {\rm Vect}(k),\,\,\ \ \{E_i\}_{i\geq 0}
\ \longmapsto\ (E_0)_x.
\end{equation}
Then the pair $(\Vect^{str}(X),\, \omega_x)$ forms a neutral Tannakian category.
Its Tannaka dual is the stratified fundamental group $\pi_1^{str}(X,\, x)$ \cite{Sa}, \cite{Gi}
(see \cite{SR}, \cite{DM}, \cite{No1}, \cite{No2} for Tannaka dual). It is shown in \cite{Sa} that
this group scheme is perfect. In \cite{Sa}, the author also gives a description of the
stratified fundamental group of abelian varieties.

Let $E_{\bullet}\,:=\,\{E_i\}_{i\ge 0}$ be a stratified vector bundle on $X$. The monodromy of
$E_{\bullet}$ is defined to be the Tannaka dual of the full Tannakian subcategory
$(E_{\bullet})^{str}$ of $\Vect^{str}(X)$ generated by $E_{\bullet}$. The monodromy of
$E_{\bullet}$ is denoted by $G^{str}_{E_{\bullet}}$.

For a strongly semistable vector bundle $V$ on $X$ such that both $c_1(V)$ and $c_2(V)$ are
numerically equivalent to zero, there
is the semistable monodromy $G^{ss}_{V}$ of $V$. More precisely, let
\begin{equation}\label{vss}
\Vect^{ss}(X)
\end{equation}
be the category of strongly semistable vector bundles $W$ on $X$ for which
both $c_1(W)$ and $c_2(W)$ are numerically equivalent to zero; morphisms are
all possible ${\mathcal O}_X$--linear homomorphisms. For a closed point 
$x\,\in\, X$, let $\omega_x\,:\, \Vect^{ss}(X)\, \longrightarrow\, \Vect(k)$ be the fiber
functor defined by $W\,\longmapsto\, W_x$. Then it is well-known that the pair $(\Vect^{ss}(X),\, 
\omega_x)$ forms a neutral Tannakian category (\cite{BH}, \cite{La}). Its Tannaka dual is the
$S$-fundamental group $\pi_1^{S}(X,\, x)$. The above group scheme
\begin{equation}\label{a2}
G^{ss}_V
\end{equation}
is the Tannaka dual of the full Tannaka subcategory
\begin{equation}\label{a3}
(V)^{ss}
\end{equation}
of $\Vect^{ss}(X)$ (see \eqref{vss}) generated by the vector bundle $V$.

Note that for a stratified vector bundle $\{E_i\}_{i\ge 0}$, if $E_0$
is strongly semistable, then every $E_i$ is strongly semistable; to see this, if
$W\, \subset\, (F^j_X)^* E_i$ violates the semistability condition of $(F^j_X)^* E_i$, then
$(F^i_X)^* W\, \subset\, (F^i_X)^* (F^j_X)^* E_i\,=\, (F^j_X)^* E_0$ violates the
semistability condition of $(F^j_X)^* E_0$. Also, for any stratified vector bundle
$\{E_i\}_{i\ge 0}$, the Chern class $c_j(E_i)$ is numerically equivalent to zero for all
$i\, \geq\, 0$ and $j\,\geq\, 1$.

The main result here is the following.

\begin{theorem}\label{thmi}
Let $E_{\bullet}\,:=\,\{E_i\}_{i\ge 0}$ be a stratified vector bundle on $X$ such that $E_0$
is strongly semistable.
\begin{enumerate}
\item For each $n\,\ge\, 0$, the functor $$\frakF_n\ :\ (E_{\bullet})^{str}\ \longrightarrow\
(E_n)^{ss},$$ that sends an object $V_{\bullet}\,=\,\{V_i\}_{i\ge 0}$ in $(E_{\bullet})^{str}$
to the vector bundle
$V_n$, is a tensor functor of Tannakian categories. The functor $\frakF_n$ induces injective 
morphism of group schemes $$\scrf_n\ :\ G^{ss}_{E_n}\ \longrightarrow\ G^{str}_{E_{\bullet}}.$$

\item For all $n\, \geq\, 0$, there are injective homomorphisms $\scrf_n\,\longrightarrow\,\scrf_{n+1}$
that make $\{ \image{\scrf_n} \}$ an increasing sequence of subgroup schemes of $G^{str}_{E_{\bullet}}$.

\item The closure of the direct limit $\lim_{n\to \infty} \scrf_n$ is $G^{str}_{E_{\bullet}}$.
\end{enumerate}
\end{theorem}

Note that there is no map between the S--fundamental group scheme $\Pi^S(X,\,x)$ and the stratified fundamental 
group scheme $\Pi^{str}(X,\,x)$, though both surject onto the \'etale fundamental group of $X$. In the process of 
computing monodromy, we also construct a group scheme $\Pi^{strss}(X,\,x)$ which is a quotient of 
$\Pi^{str}(X,\,x)$ and admits many natural morphisms of group schemes from $\Pi^S(X,\,x)$ to $\Pi^{strss}(X,\,x)$ 
(see \eqref{a6}). In Theorem \ref{thm.main} we show that the Zariski closure of the images of these morphisms is 
the whole group scheme $\Pi^{strss}(X,\,x)$.

In the last section, we show that for a smooth projective variety 
$X$ satisfying the condition that the tangent bundle $TX$ is trivial,
$$\Pi^{strss}(X,\,x)\ =\ \Pi^{str}(X,\,x)$$
(see Theorem \ref{tt}).

For an abelian variety $X$, the group scheme $\Pi^{str}(X,\,x)$ is explicitly described in \cite[Theorem 21]{Sa}.
In view of Theorem \ref{tt}, this also produces an explicit description of $\Pi^{strss}(X,\,x)$.

\section{Comparison of monodromies}

We adopt the following notation.

Take any stratified vector bundle $E_{\bullet}\,:=\,\{E_i\}_{i\ge 0}$.
For every integer $j\,<\,0$, set $E_j\,:=\,(F^{-j}_X)^* E_0$. For any $n\,\in\, \ZZ$,
denote by $E_{\bullet -n}$ the stratified bundle whose $i^{th}$ term is $E_{i+n}$
for all $i\,\ge\, 0$.

Define the function
\begin{equation}\label{ec}
C\ :\ \Vect^{str}(X)\ \longrightarrow\ \Vect^{str}(X), \ \ \, E_{\bullet}\ \longmapsto\ E_{\bullet-1}.
\end{equation}
The inverse functor that sends any $E_{\bullet}$ to $E_{\bullet +1}$ will be denoted by $\Gamma^{str}$;
note that $\Gamma^{str}$ is given by the pullback using the Frobenius morphism $F_X$, in other words,
$\Gamma^{str}(E_{\bullet})_j\,=\, F^*_X E_j$ for all $j\, \geq\, 0$.

\begin{definition}
Let $\Vect^{strss}(X)$ be the full subcategory of $\Vect^{str}(X)$ consisting of objects
$E_{\bullet}$ such that the vector bundle $E_0$ is strongly semistable.
\end{definition}

Denote the $\Vect (k)$ the category of finite dimensional vector spaces over the field $k$.
Fix a point $x\, \in\, X$. We have the fiber functor $\varpi_x\, :\, \Vect^{strss}(X)\, \longrightarrow\, \Vect (k)$
that sends any stratified vector bundle $\{E_i\}_{i\ge 0}$ to the fiber $(E_0)_x$ of $E_0$ over $x$.

\begin{proposition}\label{prop1}
The category $\Vect^{strss}(X)$ is a full Tannakian subcategory of $\Vect^{str}(X)$.
\end{proposition}

\begin{proof}
Note that for an object $V_{\bullet}$ of $\Vect^{strss}(X)$, the vector bundle $V_0$ is strongly semistable and 
the Chern classes $c_\ell(V_0)\,=\,0$ for all $\ell\,>\,0$. For strongly semistable vector bundles $V$ and $W$
on $X$, the vector bundles $V\oplus W$ and $V\otimes W$ are strongly semistable \cite[p.~288, Theorem 3.23]{RR}.
Also, the dual $V^*$ is strongly semistable. Furthermore, if $c_\ell(V)\,=\,0\, =\, c_\ell(W)$ for all
$\ell\,>\,0$, then $c_\ell (V\oplus W)\,=\,0$ and $c_\ell (V\otimes W)\,=\,0$ for all $\ell\,>\,0$.

Let $V$ and $W$ be strongly semistable vector bundles such $c_1(V)$, $c_2(V)$, $c_1(W)$ and $c_2(W)$ are all 
numerically equivalent to zero. If $\phi\, :\, V\, \longrightarrow\, W$ is any ${\mathcal O}_X$--linear 
homomorphism, then $\text{kernel}(\phi)$ is strongly semistable vector bundle, if $\text{kernel}(\phi)\, 
\not=\,0$, and in that case, $c_j(\text{kernel}(\phi))$ is numerically equivalent to zero for all $j\, \geq\, 
1$ (see \cite{BH}, \cite{La}). Similarly, $\text{cokernel}(\phi)$ is strongly semistable, and
$c_j(\text{cokernel}(\phi))$ is numerically
equivalent to zero for all $j\, \geq\, 1$, if $\text{cokernel}(\phi)\, \not=\,0$. 
The proposition follows from these.
\end{proof}

\begin{definition}\label{def1}
The proalgebraic group scheme given by the Tannaka dual of the neutral Tannakian category $(\Vect^{strss}(X),\,
\varpi_x)$ will be denoted by $\Pi^{strss}(X,\,x)$.
\end{definition}

There is a natural fully faithful tensor functor
\begin{equation}\label{if}
\Vect^{et}(X)\,\ \longrightarrow\,\ \Vect^{str}(X)
\end{equation}
(see \cite[p.~7, Proposition 1.9]{Gi}, \cite[p.~702, Proposition 13]{Sa}). The functor
in \eqref{if} factors through $\Vect^{strss}(X)$, more precisely, we have
$$
\Vect^{et}(X)\ \longrightarrow\ \Vect^{strss}(X)\ \longrightarrow\ \Vect^{str}(X).
$$
Consequently, the surjective morphism of group schemes $\Pi^{str}(X,\,x)\,\longrightarrow
\,\Pi^{et}(X,\,x)$ factors through the Tannakian dual $\Pi^{strss}(X,\,x)$ of $\Vect^{strss}(X)$, so
$$
\Pi^{str}(X,\,x)\ \longrightarrow\ \Pi^{strss}(X,\,x)\ \longrightarrow
\ \Pi^{et}(X,\,x),
$$
where both the homomorphisms are actually surjective; their surjectivity is a straightforward
consequence of the criterion in \cite[p.~139, Proposition 2.21(a)]{DM} for surjectivity.

\begin{remark}\label{rem1}\mbox{}
\begin{enumerate}
\item The group scheme $\Pi^{strss}(X,\,x)$ is a perfect scheme. It has the same proof as \cite[Theorem 
11]{Sa}.

\item Lefschetz theorem holds for $\Pi^{strss}(X,\,x)$. This is because of the
following two facts:
\begin{itemize}
\item Lefschetz theorem holds for $\Pi^{str}(X,\,x)$
(see \cite[p.~30, Theorem 5.2]{Gi}), and

\item the restriction of a strongly semistable vector bundle $W$, such that $c_1(W)$ and $c_2(W)$ are
numerically equivalent to zero, is again strongly semistable \cite{La}, \cite{BH}.
\end{itemize}
\end{enumerate}
\end{remark}

Consider $\Vect^{ss}(X)$ defined in \eqref{vss}. Let
\begin{equation}\label{a1}
\frakF\ :\ \Vect^{strss}(X)\ \longrightarrow\ \Vect^{ss}(X)
\end{equation}
be the functor that sends any $E_{\bullet} \in \Vect^{str}(X)$ to
$E_0$, and sends any morphism
$$\phi\ :=\ \{\phi_i\}_{i\ge 0}\ :\ E_{\bullet}\ \longrightarrow\ E'_{\bullet}$$ of stratified vector
bundles to the homomorphism $\phi_0\,:\, E_0\, \longrightarrow\, E_0$.

\begin{proposition}\label{prop2}
The functor $\frakF$ in \eqref{a1} is a faithful functor of Tannakian categories.
\end{proposition}

\begin{proof}
Take two objects $E_{\bullet}$ and $E'_{\bullet}$ of $\Vect^{strss}(X)$. Let
$$
\phi\, :=\, \{\phi_i\}_{i\ge 0}\, :\, E_{\bullet}\ \longrightarrow\ E'_{\bullet}\, \ \text{ and }\,\
\phi'\, :=\, \{\phi'_i\}_{i\ge 0}\, :\, E_{\bullet}\ \longrightarrow\ E'_{\bullet}
$$
be two homomorphisms such that there is some $j\, \geq\, 0$ for which $\phi_j\,=\, \phi'_j$.
Then it can be shown that
\begin{equation}\label{eh}
\phi_i\,\ =\,\ \phi'_i
\end{equation}
for all $i\, \geq\, 0$. To see \eqref{eh} note that the Frobenius morphism $F_X$
is dominant, and hence the pullback of homomorphisms of sheaves by $F_X$ is an injective
homomorphism. The proposition follows immediately from \eqref{eh}.
\end{proof}

\begin{lemma}\label{F-tannakian}
Let $V$ be an object of $\Vect^{ss}(X)$. Let $F_X$ be the absolute Frobenius morphism of
$X$. Then the following two statements hold:
\begin{enumerate}
\item The pullback $F^*_XW$ is in $(F^*_X V)^{ss}$ for every $W$ in $(V)^{ss}$ (see \eqref{a3} for
$(V)^{ss}$ and $(F^*_X V)^{ss}$).

\item The pullback $F^*_X V$ lies in $(V)^{ss}$.
\end{enumerate}
\end{lemma}

\begin{proof}
Note that if $W\,\subset \,W'$ is a subbundle, then $F^*_XW$ is a subbundle of $F^*_XW'$, and furthermore,
$F^*_X(W'/W)\,=\, (F^*_XW)/(F^*_XW')$. There are natural isomorphisms
$F^*_X(W\oplus W')\,=\, (F^*_XW)\oplus (F^*_XW')$ and
$F^*_X(W\otimes W')\,=\, (F^*_XW)\otimes (F^*_XW')$. Also, $(F_X)^*W^*\,=\, (F_X^*W)^*$.
The first statement follows from these.

To prove the second statement, recall that $F^*_X V$ is a subbundle of $V^{\otimes p}$, where $p$ is the
characteristic of the field $k$. We know that $V^{\otimes p}$ is strongly semistable, and the Chern classes
$c_j(V^{\otimes p})$ are numerically equivalent to zero for $j\,=\, 1,\,2$, because $V$ has these properties.
Since $c_j(F^*_X V)\,=\, F^*_X (c_j(V))$, it follows that $F^*_X V$ lies in $(V^{\otimes p})^{ss}$. Thus
$F^*_X V$ lies in $(V)^{ss}$.
\end{proof}

For any $n\,\ge\, 0$, denote the functor
\begin{equation}\label{fn}
\frakF_n\ :=\ \frakF\circ C^{n} \ :\ \Vect^{strss}(X)\ \longrightarrow\ \Vect^{ss}(X),
\end{equation}
where $C$ and $\frakF$ are constructed in \eqref{ec} and \eqref{a1} respectively. So,
we have $$\frakF_n(E_{\bullet})\ =\ E_n.$$

For any stratified vector bundle $E_{\bullet}\,\in\, \Vect^{strss}(X)$, let
\begin{equation}\label{a4}
(E_{\bullet})^{str}
\end{equation}
be the full Tannakian subcategory generated by $E_{\bullet}$. Denote by
\begin{equation}\label{a5}
G^{str}_{E_{\bullet}}
\end{equation}
the affine group scheme given by the neutral Tannakian category $(E_{\bullet})^{str}$ in \eqref{a4}.

\begin{proposition}\label{prop3}
For every $n\,\ge\,  0$, the functor $\frakF_n$ in \eqref{fn} gives a faithful functor of neutral
Tannakian categories
$$\Vect^{strss}(X)\ \longrightarrow\ \Vect^{ss}(X)$$ that sends any $E_{\bullet}$ to $E_n$ and any morphism
$\phi\, :=\, \{\phi_i\}_{i\geq 0}\, :\, E_{\bullet}\, \longrightarrow\, E'_{\bullet}$ to $\phi_n\,:\,
E_n\, \longrightarrow\, E'_n$.

For any stratified vector bundle $E_{\bullet}\,\in\, \Vect^{strss}(X)$, the restriction of the functor
$\frakF_n$ in \eqref{fn} to $(E_{\bullet})^{str}$ (see \eqref{a4}) is a faithful functor of Tannakian
categories $$(E_{\bullet})^{str}\,\ \longrightarrow\,\ (E_n)^{ss}$$
(see \eqref{a3}). Moreover, this restriction of the functor $\frakF_n$ induces an injective
homomorphism of group schemes $$\scrf_n\ :\ G^{ss}_{E_n}\ \longrightarrow\ G^{str}_{E_{\bullet}}$$
(see \eqref{a5} and \eqref{a2} for $G^{str}_{E_{\bullet}}$ and $G^{ss}_{E_n}$ respectively).
\end{proposition}

\begin{proof}
If $\varphi_1,\, \varphi_2\, :\, A\, \longrightarrow\, B$ are two homomorphisms of vector bundles over
$X$ such that $F^*_X\varphi_1\,=\, F^*_X\varphi_2$, then we have $\varphi_1\,=\,\varphi_2$, because
$F_X$ is dominant. Using this
it is clear that both $\frakF_n$ and the restriction of $\frakF_n$ to $(E_{\bullet})^{str}$ are
faithful functors of Tannakian categories. The injectivity of $\scrf_n$ follows from the criterion
in \cite[p.~139, Proposition 2.21(b)]{DM}.
\end{proof}

\begin{lemma}\label{F-monodromy}
Take any object $V$ of $\Vect^{ss}(X)$. As in \eqref{a2}, denote by $G^{ss}_V$ the monodromy of $V$.
Consider the Frobenius morphism $${\mathbb F}_V\,\ :\,\ G^{ss}_V\,\ \longrightarrow\,\ G^{ss}_V.$$
The monodromy group scheme $G^{ss}_{F^*_XV}$ of $F^*_XV$ is naturally isomorphic to the
image of this homomorphism ${\mathbb F}_V$.
\end{lemma}

\begin{proof}
The functor $F^*_X$ from $\Vect^{ss}(X)$ to itself is a faithful functor of Tannakian categories. The induced
endomorphism of the Tannaka dual $$\Phi\ :\ \Pi^S(X,\,x)\ \longrightarrow\ \Pi^S(X,\,x)$$ is the absolute
Frobenius morphism. The same holds for the full subcategory $(V)^{ss}$ generated by $V$ which induces
absolute Frobenius morphism $${\mathbb F}_V\ :\ G^{ss}_V\ \longrightarrow\ G^{ss}_V.$$

Let ${\mathbb G}_V$ denote the subgroup scheme $Im({\mathbb F}_V)\, \subset\, G^{ss}_V$. The morphism
\begin{equation}\label{b1}
\widetilde{\mathbb F}_V\ :\ G^{ss}_V\ \longrightarrow\ {\mathbb G}_V
\end{equation}
given by ${\mathbb F}_V$ coincides with the one induced by the Tannakian full 
subcategory
\begin{equation}\label{ec2}
\scrC
\end{equation}
of $(V)^{ss}$ generated by objects $F^*_XW$ where $W$ runs over $(V)^{ss}$ (i.e., the image of 
the functor $F^*_X$).

To prove that ${\mathbb G}_V$ is naturally isomorphic to $G^{ss}_{F^*_XV}$, it
suffices to show that the full Tannakian subcategory of $(V)^{ss}$ generated by $F^*_XV$ is precisely $\scrC$
in \eqref{ec2} (see Lemma \ref{F-tannakian}(2)). Since $F^*_XV$ is an object of $\scrC$, it is enough
to show that $F^*_XW$ is in $(F^*_XV)^{ss}$ for every $W$ in $(V)^{ss}$. But this 
follows from Lemma \ref{F-tannakian}(1).
\end{proof}

\begin{theorem}\label{th1}
Let $E_{\bullet}$ be an object in ${\rm Vect}^{strss}(X)$. For each $i\, \geq\, 1$, the following diagram
is commutative:
 \[ \xymatrix{
 & G^{str} _{E_{\bullet}} & \\
 G^{ss} _{E_i}\ar@{->>}[r]^{\widetilde{\mathbb F}_{E_i}}\ar@/_1.5pc/[rr]_{\mathbb F_{E_i}}
& G^{ss}_{E_{i-1}}\ar[u]^{\scrf_{i-1}}
\ar@{^{(}->}[r] & G^{ss} _{E_i}\ar[ul]_{\scrf_i}
  }
 \]
where $\widetilde{\mathbb F}_{E_i}$ and
$\scrf_j$ are the homomorphisms constructed in \eqref{b1} and Proposition \ref{prop3} respectively, and
${\mathbb F}_{E_i}$ is the Frobenius morphism of $G^{ss}_{E_i}$.
\end{theorem}

\begin{proof}
Since $F^*_XE_i\,\cong\, E_{i-1}$, by Lemma \ref{F-monodromy} we get that $G^{ss}_{E_{i-1}}$ is the
image of the Frobenius morphism on $G^{ss}_{E_i}$. This yields the following commutative diagram:
 \[ \xymatrix{
 G^{ss} _{E_i}\ar@{->>}[r]\ar@/_1.5pc/[rr]_{{\FF}_{E_i}} & G^{ss}_{E_{i-1}}\ar@{^{(}->}[r] & G^{ss} _{E_i};
  }
 \]
the surjective map on the left is the map ${\FF}_{E_i}\, :\, G^{ss} _{E_i}\, \longrightarrow\,
{\rm Image}({\FF}_{E_i})\,=\, G^{ss}_{E_{i-1}}$, and the injective map on the right is the
inclusion map $G^{ss}_{E_{i-1}}\,=\,{\rm Image}({\FF}_{E_i})\, \longrightarrow\, G^{ss} _{E_i}$.

The group scheme $G^{ss}_{E_{i-1}}$ is the Tannaka dual of the Tannakian subcategory $(E_{i-1})^{ss}$ of $(E_i)^{ss}$
generated by all $F^*_XV$ with $V$ in $(E_i)^{ss}$. Also, pullback by the Frobenius
morphism $F_X$ is the functor of Tannakian categories from $(E_i)^{ss}$ to $(E_{i-1})^{ss}$ that
induce the inclusion of the Tannaka duals
\begin{equation}\label{ti}
\iota\ :\ G^{ss}_{E_{i-1}}\ \hookrightarrow\ G^{ss} _{E_i}.
\end{equation}

Since $\scrf_i$ is induced by the functor $\frakF_i$ and $\frakF_{i-1}\,=\,F^*_X\circ \frakF_i$, passing
to Tannaka duals, we obtain the equality $$\scrf_{i-1}\ =\ \scrf_i\circ \iota,$$
where $\iota$ is the map in \eqref{ti}. This completes the proof.
\end{proof}

The functor $\frakF_i\,: \,\Vect^{strss}(X)\,\longrightarrow\, \Vect^{ss}(X)$ between
neutral Tannakian categories (see \eqref{fn}) produces a homomorphism
\begin{equation}\label{a6}
\widehat{\frakF}_i\ :\ \Pi^S(X,\,x)\ \longrightarrow\ \Pi^{strss}(X,\,x)
\end{equation}
between the corresponding proalgebraic group schemes. Let
\begin{equation}\label{a7}
F_{\Pi^{S}}\ :\ \Pi^{S}(X,\, x)\ \longrightarrow\ \Pi^{S}(X,\, x)
\end{equation}
be the Frobenius morphism on $\Pi^{S}(X,\, x)$.

\begin{lemma}\label{lem-a}
The morphisms of group schemes in \eqref{a6} and \eqref{a7} satisfy the following:
$$\widehat{\frakF}_i \circ F_{\Pi^{S}}\,\ = \,\ \widehat{\frakF}_{i-1}.$$
\end{lemma}

\begin{proof}
The proof of the lemma is exactly the same as the proof of Theorem \ref{th1}.
\end{proof}

{}From Theorem \ref{th1} it follows that we can take the direct limit of
$\image{\scrf_n}$ as $n\,\to\, \infty$. In view of Lemma \ref{lem-a} we can take the
direct limit of the subgroups $\image{\widehat \frakF_n}\, \subset\, \Pi^{strss}(X,\,x)$
as $n\, \to\, \infty$.

\begin{theorem}\label{thm.main}
The direct limit $\lim_{n\ge 0} \image{\scrf_n}$ (see Proposition \ref{prop3})
is Zariski dense in $G^{str}_{E_{\bullet}}$ (see \eqref{a5}).
Similarly, the Zariski closure of the direct limit $\lim_{n\ge 0} \image{\widehat \frakF_n}$
(see \eqref{a6}) in $\Pi^{strss}(X,\,x)$ is the entire group scheme $\Pi^{strss}(X,\,x)$.
\end{theorem}

\begin{proof}
Note that the morphism $\scrf_n\,:\,G^{ss} _{E_n}\longrightarrow G^{str}_{E_{\bullet}}$ is induced by the functor $\frakF_n$ restricted to the full Tannakian subcategory $(E_{\bullet})^{str}$
of $\Vect^{strss}(X)$ generated by $E_{\bullet}$ (see Proposition \ref{prop3}). It is
clear that the restriction of $\frakF_n$ is a faithful functor from $(E_{\bullet})^{str}$ to the category $(E_n)^{ss}$. We use $\frakF_n$ to also denote this restriction of the functor $\frakF_n$ to the subcategory $(E_{\bullet})^{str}$. 
Let $H$ denote the Zariski closure of $\lim_{n\ge 0} \image{\scrf_n}$ in $G^{str}_{E_{\bullet}}$.
 For every $n\,\ge\, 0$, there are morphisms
 $$G^{ss}_{E_n}\ \longrightarrow\ H\ \longrightarrow\ G^{str}_{E_{\bullet}}$$ that induce faithful functors
 $$(E_{\bullet})^{str}\ \xrightarrow{\,\,\,\Phi\,\,\,}\ Rep_H \ \xrightarrow{\,\,\,\Psi_n\,\,\,}
\, (E_n)^{ss}$$
such that $\Psi_n\circ\Phi\,=\, \frakF_n$.
 Also, since $F_X^*\circ \frakF_n\,=\,\frakF_{n-1}$, it follows that $F_X^*\circ \Psi_n\,=\,\Psi_{n-1}$.

 Let $V_{\bullet}$ and $W_{\bullet}$ be two objects in $(E_{\bullet})^{str}$, and $\alpha
\,\in\, Hom_H(\Phi V_{\bullet},\,\Phi W_{\bullet})$.
We have $\Psi_n(\alpha)\,=\,F_X^{\ell*}\circ\Psi_{n+\ell}(\alpha)$ for all $\ell\,\ge\, 0$. In other words,
 $$\Psi_n(\alpha)\ =\ F_X^*\circ\Psi_{n+1}(\alpha)
$$
for all $n\,\ge\, 0$. This means that $\alpha_{\bullet}\,:=\,(\Psi_n(\alpha))_{n\ge 0}$ is in
$Hom(V_{\bullet},\,W_{\bullet})$ and $\Phi(\alpha_{\bullet})\,=\,\alpha$.
 Hence $\Phi$ is a fully faithful functor. Consequently, the induced homomorphism $H\,
\longrightarrow\, G^{str}_{E_{\bullet}}$ is also surjective.

The proof of the second part is similar.
\end{proof}

\begin{corollary}
 Let $E_{\bullet}$ be an object in $\Vect^{strss}(X)$. The monodromy $G^{str}_{E_{\bullet}}$ is connected if
and only if $G^{str}_{E_{\bullet}}$ is the Zariski closure of the union of the connected components,
containing the identity element, of group schemes $\scrf_n(G^{ss}_{E_n})$ for $n\,\ge\, 0$.
\end{corollary}

\begin{proof}
Since $\scrf_n(G^{ss}_{E_n})$ is an increasing sequence of group schemes, the union of their connected 
components is connected. Hence the Zariski closure is also connected.
\end{proof}

\section{Some examples}

Let $A$ be an irreducible smooth projective variety such that the tangent bundle $TA$ is trivial.
Abelian varieties satisfy this condition; see \cite{Ke} for other examples.

\begin{theorem}\label{tt}
The natural homomorphism
$$\Pi^{str}(A)\ \longrightarrow\ \Pi^{strss}(A)$$ is an isomorphism.
\end{theorem}

\begin{proof}
It is enough to show that for every stratified vector bundle $E_{\bullet}$ on $A$, the
vector bundle $E_0$ is strongly semistable.

For a semistable vector bundle $V$ on the variety $A$, it can be shown that the
Frobenius pullback $F^*_A V$ is also
semistable. Indeed, if $F^*_A V$ fails to be semistable, consider the maximal semistable subsheaf
$$
W \ \, \subset\ \, F^*_A V
$$
of maximal slope; so $W$ is the first nonzero term of the Harder--Narasimhan filtration of $F^*_A V$.
Since there is no nonzero homomorphism from $W$ to $(F^*_A V)/W$, and
$$\text{Hom}(W,\, (F^*_A V)/W)\otimes\Omega^1_A\ =\
\text{Hom}(W,\, (F^*_A V)/W)\otimes {\mathcal O}^{\oplus d}_A,$$
where $d\,=\,\dim A$ (recall that $TA$ is trivial), we have
\begin{equation}\label{ev}
H^0(A,\, \text{Hom}(W,\, (F^*_A V)/W)\otimes\Omega^1_A)\ =\ 0.
\end{equation}

Let ${\mathbb D}\, :\, F^*_A V\, \longrightarrow\, (F^*_A V)\otimes \Omega^1_A$ 
be the Cartier connection on $F^*_A V$. Consider the following composition of homomorphisms
$$
W \ \hookrightarrow\ F^*_A V \ \stackrel{\mathbb D}{\longrightarrow}\ (F^*_A V)\otimes \Omega^1_A
\ \longrightarrow\ ((F^*_A V)/W)\otimes \Omega^1_A.
$$
{}From \eqref{ev} it follows that this composition of homomorphisms vanishes identically.
Thus the Cartier connection on $F^*_A V$ preserves $W$.

Hence there is a subsheaf $W'\, \subset\, V$
such that $F^*_A W'\,=\, W$. But $W'$ violates the semistability condition for $V$ because $W$
violates the semistability condition for $F^*_A V$. Thus $F^*_A V$ is semistable if $V$ is so.

Take a stratified vector bundle $E_{\bullet}$ on $A$. So $E_m$ is semistable for all $m$ sufficiently large.
Take such a $m$. Therefore,
$$
(F^{m+n}_A)^* E_m \ =\ (F^n_A)^* ((F^m_A)^* E_m) \ =\  (F^n_A)^* E_0
$$
is semistable for all $n$. Thus $E_0$ is strongly semistable.
\end{proof}

\begin{example}
Let $X$ be an elliptic curve and $L_{\bullet}$ be the stratified vector bundle on $X$ such that $L_n$ is
a nontrivial bundle and $F^{n*}_X L_n$ is trivial for $n\,>\,1$. Then $G^{str}_{L_{\bullet}}$ is the
torus $\GG_m$. This is because $G^{ss}_{L_n}$ is $\mu_{p^n}$ sitting inside $\GG_m$. Since
$G^{str}_{L_{\bullet}}$ is the subgroup scheme of $\GG_m$ containing the Zariski closure of the union
of $\mu_{p^n}$, it must be the full $\GG_m$.
\end{example}

\begin{example}
We will show that the natural homomorphism $$\Pi^{str}(X,\, x)\ \longrightarrow\ \Pi^{strss}(X,\, x)$$ is, in 
general, not an isomorphism. Gieseker constructed examples of stratified vector bundles $\{E_i\}_{i\geq 0}$ such 
that $E_0$ is not semistable \cite{Gi1}. In his example, $X$ is a smooth projective curve of genus at least two, 
and $E_0$ is a nontrivial extension of $K^{-1/2}_X$ by $K^{1/2}_X$, where $K^{1/2}_X$ is a theta characteristic 
(a square-root of the canonical line bundle); see \cite[p.~99, Lemma 4]{Gi1} and \cite[p.~99, Proposition 
2]{Gi1}. In view of the existence of such examples it follows immediately that the natural homomorphism 
$\Pi^{str}(X,\, x)\, \longrightarrow\, \Pi^{strss}(X,\, x)$ is, in general, not an isomorphism.
\end{example}

\end{document}